\documentclass[12pt]{article}

\usepackage[T1]{fontenc}
\usepackage[active]{srcltx}
\usepackage{lmodern}
\usepackage{amsmath,amstext,amssymb,graphicx,graphics,color}
\usepackage{mathtools,mathrsfs}
\usepackage{tikz}
\usetikzlibrary{decorations.markings,calc,fit,intersections,arrows,matrix,trees,positioning,shapes,tikzmark,backgrounds}
\usepackage{mleftright}
\usepackage{pgfplots}
\usepackage{theorem}
\usepackage{cite}               
\usepackage{hyperref}           
\usepackage{bbm}                
\usepackage{fancybox}           
\usepackage[section]{placeins}  
\usepackage[font=footnotesize]{caption}
\usepackage{subcaption}

\DeclarePairedDelimiter\floor{\lfloor}{\rfloor}

\usepackage{arydshln} 
\tikzset{
    invisible/.style={opacity=0},
    visible on/.style={alt={#1{}{invisible}}},
    alt/.code args={<#1>#2#3}{%
      \alt<#1>{\pgfkeysalso{#2}}{\pgfkeysalso{#3}}%
  }
}
\theorembodyfont{\normalfont\slshape} 

\newtheorem{theorem}{Theorem}
\newtheorem{proposition}{Proposition}[section]
\newtheorem{corollary}[proposition]{Corollary}

\newtheorem{lemma}[proposition]{Lemma}
\theoremstyle{break} 

\newenvironment{proof}%
{{\par\noindent \bf Proof. \nobreak}}%
{\nobreak \removelastskip \nobreak \hfill $\Box$ \medbreak}
\newenvironment{remark}{\par \medskip \noindent {\bf Remark. }\nobreak}{\par \medskip}

\usepackage{fancyhdr}
 \fancypagestyle{plain}{
  \fancyhead{}
  
}

\def\paragraph#1{{\bf #1\ }}

\newcommand{\expo}{\mathrm{e}}

\newcommand{\Var}{\mathrm{Var}}

\newcommand{\dd}{\mathrm{d}}
\newcommand{\EE}{\mathbb{E}}

\newcommand{\NN}{\mathrm{N}}

\usepackage[utf8]{inputenc}
\usepackage{unicode_character}

\def\law{\operatorname{Law}}

\title{Fractal opinions among interacting agents} 

\author{Fei Cao \footnotemark[1] \and Roberto Cortez \footnotemark[2]}

\begin{document}
\maketitle

\footnotetext[1]{University of Massachusetts Amherst - Department of Mathematics and Statistics, Amherst, MA 01003, USA}
\footnotetext[2]{Universidad Andrés Bello - Departamento de Matemáticas, Santiago, Chile}

\begin{abstract}
We investigate an opinion model consisting of a large group of interacting agents, whose opinions are represented as numbers in $[-1,1]$. At each update time, two random agents are selected, and the opinion of the first agent is updated based on the opinion of the second (the ``persuader''). We derive the mean-field kinetic equation describing the large population limit of the model, and we provide several quantitative results establishing convergence to the unique equilibrium distribution. Surprisingly, in some range of the model parameters, the support of the equilibrium distribution exhibits a fractal structure. This provides a new mathematical description for the so-called opinion fragmentation phenomenon.
\end{abstract}


\tableofcontents

\section{Introduction}\label{sec:sec1}
\setcounter{equation}{0}

\subsection{Model description}

In this work we study a model of opinion dynamics among interacting agents, both for the system of finitely many agents and its large population limit. In the finite case, there are $N$ agents (labelled from $1$ to $N$) located on a complete graph, where each agent $i$ is characterized uniquely by a number $X^{i,N} \in [-1,1]$, representing his/her general \emph{opinion} or \emph{political standpoint} on a given topic. Hence, one can interpret $-1$ and $+1$ as extreme {\it left-wing} and extreme {\it right-ring}, respectively. Each individual interacts with all other individuals at rate one. The dynamic evolves according to the following rule: at each random time (generated by a Poisson clock with rate $N/2$), a pair of distinct agents $(i,j) \in \{1,\cdots,N\}^2 \setminus \{i=j\}$ is picked independently and uniformly at random, then we update the opinion level of agent $i$ according to
\begin{equation}
\label{eq:dynamics}
X_t^{i,N} = \begin{cases}
X_{t^-}^{i,N} + \mu_+\cdot\left(1 - X_{t^-}^{i,N} \right) &~~ \textrm{with probability}~~ \frac 12 + \frac{X_{t^-}^{j,N}}{2}, \\
X_{t^-}^{i,N} - \mu_{-}\cdot\left(1 + X_{t^-}^{i,N} \right) &~~ \textrm{with probability}~~ \frac 12 - \frac{X_{t^-}^{j,N}}{2},
\end{cases}
\end{equation}
where $\mu_{-} \in (0,1]$ and $\mu_+ \in (0,1]$ are two parameters modeling the velocity/tendency of the individuals towards $-1$ and $+1$, respectively. This rule can be seen as agent $j$ (the ``persuader'') stating an opinion, which can be either $-1$ or $+1$ with probabilities depending linearly on his/her current political standpoint (i.e., given by $(1 - X_{t^-}^{j,N})/2$ and $(1 + X_{t^-}^{j,N})/2$, respectively), and then agent $i$ updates his/her standpoint by moving a proportion ($\mu_-$ or $\mu_+$) toward the stated opinion.

Without loss of generality and by the obvious symmetry we will assume that
\[
	\mu_{-} \leq \mu_{+}
\]
unless otherwise stated. We also emphasize that if we put $X_0^{i,N} \in \{-1,1\}$ for all $1\leq i\leq N$ initially and set $\mu_{-} = \mu_+ = 1$, then the aforementioned model boils down to the classical voter model \cite{clifford_model_1973,holley_ergodic_1975} (also known as the Moran model in evolutionary biology and population genetics literature \cite{etheridge_some_2011}).

The literature on opinion models is vast \cite{cao_k_2021,cao_iterative_2024,castellano_statistical_2009,jabin_clustering_2014,sen_sociophysics_2014,weber_deterministic_2019} and we do not aim to provide a comprehensive list of reviews on opinion dynamics and the related sociophysics models. Broadly speaking, the Deffuant (bounded confidence) model \cite{deffuant_mixing_2000}, the Krause-Hegselmann model \cite{hegselmann_opinion_2002}, and the Sznajd model \cite{sznajd_opinion_2000} and their variants are among the most popular models studied in the literature on opinion dynamics. One important inspiration behind the present work is the very recent research by Nicolas Lanchier and Max Mercer \cite{lanchier_limiting_2024}, in which the model under our investigation in this manuscript was first proposed (although in a discrete time setting) using a different set of terminologies. To be precise, they term $X^{i,N} \in [-1,1]$ as the \emph{kindness} level of agent $i$ and interpret the rule \eqref{eq:dynamics} as the update of agent $i$'s kindness after a kind or unkind interaction with another agent $j$, whence in their settings the model parameters $\mu_{-} \in [0,1]$ and $\mu_+ \in [0,1]$ measure the sensitivity of individuals to a unkind/kind interaction, respectively.

Denote ${\bf X}_t^N = (X_t^{1,N},\ldots,X_t^{N,N})$ and assume that $\mu_- < \mu_+$. The main result proved in \cite{lanchier_limiting_2024} using standard tools from probability theory and Markov chains \cite{lanchier_stochastic_2017,liggett_interacting_1985} can be briefly summarized into the following (informal) statement: suppose that $X_0^{i,N} \sim \textrm{Uniform}[-1,1]$ for all $1\leq i\leq N$ and they are independent, then
\[
	\mathbb{P}\left(A_1 \cup A_{-1}\right) = 1,
\]
where $A_1$ and $A_{-1}$ represent the event that ${\bf X}_t^N \xrightarrow{t\to \infty} (1,1,\ldots,1)$ and ${\bf X}_t^N \xrightarrow{t\to \infty} (-1,-1,\ldots,-1)$, respectively. Moreover, under the large population limit $N \to \infty$, the authors of \cite{lanchier_limiting_2024} also proved that
\[
\mathbb{P}\left(\textrm{event $A_{-1}$ happens before event $A_1$ happen}\right) = 0.
\]

One of our main goals in this manuscript is to analyze the model using a kinetic approach, i.e., we will first take the large population limit $N \to \infty$ to obtain a Boltzmann-type partial differential equation (PDE), then investigate the long time behavior of the resulting mean-field PDE as $t \to \infty$. As for the underlying $N$-agent system (running in continuous time), we resort to a stochastic differential equation (SDE) framework (see for instance \cite{cortez_quantitative_2016,cortez_uniform_2016}) using Poisson point measures to demonstrate various analytical results of the model, including the surprising emergence of certain fractal structures in the equilibrium distribution of opinions

\subsection{Main results and plan of the paper}

The rest of this paper is organized as follows. Section \ref{sec:sec2} is devoted to the introduction/derivation of the rigorous mean-field limit of the interacting random opinion dynamics \eqref{eq:dynamics} as the number of agents $N$ tends to infinity. We deduce both the SDE for the limit process and the PDE describing the evolution of its marginal distributions; the computation of the first moment along the solution to the mean-field PDE can be carried out readily.

We investigate the large time asymptotic behavior of this mean-field Boltzmann-type equation in Section \ref{sec:sec3}. In particular, when $\mu_{-} < \mu_+$, we prove that the solution of the mean-field dynamics converges to a Dirac delta at $+1$ exponentially fast in the Wasserstein metric of order $2$, see Theorem \ref{thm:1}. On the other hand, when $\mu_{-} = \mu_+$, the equilibrium distribution is no longer a Dirac delta, and we prove convergence to it with explicit exponential decay rates under various metrics, see Theorems \ref{thm:2} and \ref{thm:3}.

In Section \ref{sec:sec4} we study the properties of the stationary distribution in the special case in which the two model parameters $\mu_{\pm}$ are identical, denoted by $\mu$, and the mean opinion is zero. We characterize the equilibrium through an equation in distribution and provide some semi-explicit formulas for it. We demonstrate a remarkable phenomenon: when $\mu > 1/2$, the equilibrium of opinions is uniformly distributed on a fractal, Cantor-like set, see Proposition \ref{prop:Cantor-like_equili}. This surprising result provides a mathematically rigorous explanation for the emergence of the so-called \emph{fragmentation} or \emph{polarization} of (public) opinions observed in a number of recent reports \cite{meng_disagreement_2023,su_robust_2019}. Loosely speaking, opinion fragmentation entails the ``clusterization'' of opinions; that is, the tendency of the agents' opinions to accumulate rather than spread continuously over $[-1,1]$. For instance, in the bounded confidence model \cite{deffuant_mixing_2000}, the equilibrium distribution is a sum of one or more Dirac masses, whose exact locations depend on the initial distribution. In contrast, for the model of the present article, there exists a unique equilibrium attracting all initial distributions, whose support exhibits ``holes'' in the continuous opinion spectrum, when $\mu>1/2$. In other words: in the long run, it is impossible for agents to hold opinion values in certain intervals, given by the complement of the aforementioned Cantor-like set. This provides an alternative mathematical description of the fragmentation phenomenon. As for the case $\mu \in (0,1/2]$, the problem of finding the explicit form of the equilibrium distribution is much harder in general. We give explicit identifications (in closed form) for two specific values of $\mu$. We also provide some numerical simulations displaying the shape of the equilibrium distribution for some specific values of $\mu$. We then prove a quantitative convergence guarantee for the standardized equilibrium distribution towards a standard normal distribution $\mathcal{N}(0,1)$ when $\mu \to 0$, see Theorem \ref{prop:Gaussian_approximation}.

Thus, one of our main contributions is that our work bridges interacting multi-agent systems and kinetic-type equations with opinion models and with the so-called \emph{Bernoulli convolution} literature (to be explained later). This is highlighted in Section \ref{sec:sec5}, where we also sketch several possible directions for future research endeavors. Lastly, we provide the proof of some of our results in the Appendix.

\section{Derivation of the mean-field limit}\label{sec:sec2}
\setcounter{equation}{0}

The \emph{mean-field limit} of the system ${\bf X}_t^N$ represents the behaviour of any agent $X^{i,N}$ in the large population limit $N \to \infty$. It can refer to either a process $(Z_t)_{t\geq 0}$ or its collection of marginal distributions $(\rho_t \coloneqq \law(Z_t))_{t\geq 0}$. The former is described by a jump SDE, while the latter is described by an integro-differential PDE.

Formally, to go from the $N$-agent system to the mean-field limit, in the interaction rule \eqref{eq:dynamics} one replaces $X_t^{j,N}$ by an independent sample of $\rho_t$. More specifically, the process $(Z_t)_{t\geq 0}$ solves the jump SDE
\begin{equation}
	\label{eq:SDE}
	\dd Z_t = \int_{-1}^1 \int_0^1 \left[  \mu_{+} (1-Z_{t^-}) \mathbbm{1} \{ u < \tfrac{1+x}{2}\}
					- \mu_{-} (1+Z_{t^-}) \mathbbm{1} \{ u \geq \tfrac{1+x}{2}\}
				\right] \tilde{\mathcal{P}}(\dd t, \dd u, \dd x),
\end{equation}
where $\tilde{\mathcal{P}}(\dd t, \dd u, \dd x)$ is a Poisson point measure on $[0,\infty) \times [0,1] \times [-1,1]$ with intensity $\dd t \, \dd u \, \rho_t (\dd x)$. The fact that this intensity depends on the law of the process makes the equation nonlinear.

The rigorous convergence as $N \to \infty$ of a finite system towards its mean-field limit is known as \emph{propagation of chaos} \cite{sznitman_topics_1991}. It has been studied extensively for a huge variety of systems (ranging from social-economical sciences to life and physical sciences \cite{cao_derivation_2021,cao_entropy_2021,cao_explicit_2021,cao_interacting_2022,cao_uncovering_2022,cao_uniform_2024,matthes_steady_2008,naldi_mathematical_2010}), especially in the context of kinetic Boltzmann-type equations \cite{degond_macroscopic_2004}. For the model of the present article, propagation of chaos can be easily obtained as a consequence of well-established results. For instance, the proof of the following result is a straightforward application of \cite[Theorem 3.1]{graham_meleard_1997}:

\begin{proposition}
	The SDE \eqref{eq:SDE} admits a unique (in law) solution. Moreover, assuming that $X_0^{1,N}, \ldots, X_0^{N,N}$ are i.i.d.\ and $\rho_0$-distributed, then we have propagation of chaos: for any fixed $k$ and $T>0$,
	\[
	\lim_{N \to \infty} \law\left( (X_t^{1,N}, \ldots,X_t^{k,N})_{t \in [0,T]} \right)
	= \law\left((Z_t)_{t \in [0,T]} \right)^{\otimes k}
	\qquad
	\text{weakly.}
	\]
\end{proposition}

From the SDE \eqref{eq:SDE}, we can easily obtain the following PDE in weak form for $(\rho_t)_{t \geq 0}$, which is nothing but the Kolmogorov backward equation of the process $Z_t$: for any test function $\varphi$,
\begin{equation}
\label{eq:PDE_weak}
\begin{split}
&\frac{\dd}{\dd t} \int_{-1}^1 \varphi(x) \rho_t(\dd x) \\
&= \int_{-1}^1 \left[ \varphi(x + \mu_+ (1-x)) \frac{1+m_t}{2}
+ \varphi(x - \mu_- (1+x)) \frac{1-m_t}{2}
- \varphi(x) \right] \rho_t(\dd x), \\
&\coloneqq \int_{-1}^1 Q_t[\varphi](x)\,\rho_t(\dd x)
\end{split}
\end{equation}
where we introduced the operator $Q_t[\cdot]$ defined via
\[
Q_t[\varphi](x) \coloneqq \varphi(x + \mu_+ (1-x))\, \frac{1+m_t}{2} + \varphi(x - \mu_- (1+x))\,\frac{1-m_t}{2} - \varphi(x)
\]
for all $(x,t) \in [-1,1] \times [0,\infty)$, and $m_t$ is defined by
\[
	m_t \coloneqq \EE[Z_t] = \int_{-1}^1 x \rho_t(\dd x).
\]
Moreover, $m_t$ can be computed explicitly: taking $\varphi(x) = x$ in \eqref{eq:PDE_weak} yields
\begin{align*}
\frac{\dd}{\dd t} m_t
&= (m_t + \mu_+ (1-m_t)) \frac{1+m_t}{2}
+ (m_t - \mu_- (1+m_t)) \frac{1-m_t}{2}
- m_t \\
&= -\frac{\mu_+ - \mu_{-}}{2}\,(m_t^2 - 1),
\end{align*}
which solves to
\begin{equation}
\label{eq:mt}
m_t = 1 - \frac{2}{1+\frac{1+m_0}{1-m_0}\,\expo^{(\mu_+ - \mu_{-}) t}}.
\end{equation}
Since $m_t$ is explicit, we can write the following alternative SDE for $Z_t$, equivalent (in law) to \eqref{eq:SDE}:
\begin{equation}
	\label{eq:SDE_linear}
	\dd Z_t = \int_0^1 \left[  \mu_{+} (1-Z_{t^-}) \mathbbm{1} \{ u < \tfrac{1+m_t}{2}\}
	- \mu_{-} (1+Z_{t^-}) \mathbbm{1} \{ u \geq \tfrac{1+m_t}{2}\}
	\right] \mathcal{P}(\dd t, \dd u),
\end{equation}
where $\mathcal{P}(\dd t, \dd u)$ is a Poisson point measure on $[0,\infty) \times [0,1]$ with intensity $\dd t \, \dd u$. Notice that the PDE \eqref{eq:PDE_weak} and the SDE \eqref{eq:SDE_linear} are \emph{linear} and \emph{time-inhomogeneous}, rather than nonlinear.

\begin{remark}
One can formally obtain an equation for the density of $\rho_t$ (assuming it exists), which we denote $\rho_t(x)$, slightly abusing notation. Indeed, noticing that
\begin{equation*}
\int_{-1}^1 \varphi\left(x+\mu_+\cdot(1-x)\right)\,\rho_t(x)\,\dd x = \int_{2\mu_+-1}^1 \varphi(y)\,\rho_t\left(\frac{y-\mu_+}{1-\mu_+}\right)\,\frac{\dd y}{1-\mu_+}
\end{equation*}
and
\begin{equation*}
\int_{-1}^1 \varphi\left(x-\mu_{-}\cdot(1+x)\right)\,\rho_t(x)\,\dd x = \int_{-1}^{1-2\mu_{-}} \varphi(y)\,\rho_t\left(\frac{y+\mu_{-}}{1-\mu_{-}}\right)\,\frac{\dd y}{1-\mu_{-}},
\end{equation*}
we deduce that the evolution of $\rho_t$ is governed by the following Boltzmann-type PDE (to be understood in the weak sense):
\begin{equation}
\label{eq:PDE}
\partial_t \rho_t(x) = Q^*[\rho_t](x)
\end{equation}
in which
\begin{equation}
\label{eq:operator_Q*}
\begin{aligned}
Q^*[\rho_t](x) &= \frac{1+m_t}{2}\,\frac{\mathbbm{1}\{x> 2\mu_+ -1\}}{1-\mu_+}\,\rho_t\left(\frac{x-\mu_+}{1-\mu_+}\right) \\
&\quad + \frac{1-m_t}{2}\,\frac{\mathbbm{1}\{x\leq 1-2\mu_{-}\}}{1-\mu_{-}}\,\rho_t\left(\frac{x+\mu_{-}}{1-\mu_{-}}\right) - \rho_t(x).
\end{aligned}
\end{equation}
\end{remark}

\section{Large time analysis of the mean-field dynamics}
\label{sec:sec3}
\setcounter{equation}{0}

We now turn to the asymptotic analysis of the solution $\rho_t$ to \eqref{eq:PDE_weak} as $t \to \infty$. We start with the following useful computation. Denote
\[
q_t
\coloneqq \EE[Z_t^2]
= \int_{-1}^1 x^2\,\rho_t(\dd x).
\]

\begin{lemma}
\label{lem:qt}
Let
\begin{align*}\label{eq:notations}
\alpha_t
&\coloneqq \frac{(1-\mu_+)^2 - (1-\mu_{-})^2 }{2} m_t
+ \frac{ (1-\mu_+)^2 + (1-\mu_{-})^2}{2} - 1 \\
\beta_t
&\coloneqq \frac{1+m_t}{2}\left(\mu^2_+ + 2(1-\mu_+)\mu_+ m_t \right) + \frac{1-m_t}{2}\left(\mu_-^2 - 2(1-\mu_{-})\mu_{-} m_t \right).
\end{align*}
Call $I_t \coloneqq \expo^{-\int_0^t \alpha_s \,\dd s}$. Then \begin{equation}
\label{eq:2ndmoment}
q_t
= \frac{q_0}{I_t}
+ \frac{1}{I_t} \int_0^t \beta_s\, I_s\, \dd s.
\end{equation}
\end{lemma}

\begin{proof}
Taking $\varphi(x) = x^2$ in \eqref{eq:PDE_weak} gives
\begin{align*}
\frac{\dd}{\dd t} q_t
&= \left[(1-\mu_+)^2 q_t + 2 (1-\mu_+)\mu_+ m_t + \mu_+^2 \right] \frac{1+m_t}{2} \\
& \qquad  {} + \left[(1-\mu_-)^2 q_t - 2 (1-\mu_-)\mu_- m_t + \mu_-^2 \right] \frac{1-m_t}{2} - q_t \\
&= \alpha_t q_t + \beta_t .
\end{align*}
Solving this ODE explicitly leads us to the expression \eqref{eq:2ndmoment} for $q_t$.
\end{proof}

We now tackle the simpler case where $\mu_- < \mu_+$. Denote $W_p(\cdot,\cdot)$ to be the $p$-Wasserstein distance between probability measures on $[-1,1]$.

\begin{theorem}
	\label{thm:1}
	Assume $\mu_- < \mu_+$ and that $\rho_0$ is not the Dirac mass at $-1$ (or equivalently that $m_0 \neq -1$). Then,
	\[
	m_t \xrightarrow{t\to \infty} 1
	\qquad \text{and} \qquad
	\Var[\rho_t] \coloneqq q_t - m_t^2 \xrightarrow{t\to \infty} 0
	\]
	exponentially fast. Consequently, $\rho_t$ converges to the Dirac mass at $1$ exponentially fast in $W_2$. Moreover, the same is true for $W_1$, with the following explicit estimate:
	\[
	W_1( \rho_t , \delta_1)
	\leq 2 \frac{1-m_0}{1+m_0} \expo^{-(\mu_+ - \mu_{-}) t}.
	\]
\end{theorem}

\begin{proof}
The claim that $m_t \to 1$ exponentially fast follows directly from \eqref{eq:mt}. Consequently, using the notation of Lemma \ref{lem:qt}, we see that
\[
\lim_{t\to \infty} \alpha_t
= - \lim_{t\to \infty} \beta_t
= \mu_+^2 - 2\mu_+ < 0,
\]
which implies that $I(t) = \expo^{-\int_0^t \alpha_s \dd s} \to \infty$. From \eqref{eq:2ndmoment} and L'H\^opital's rule, we deduce that
\[
\lim_{t \to \infty} q_t
= \lim_{t \to \infty} \frac{\beta_t I_t}{\frac{\dd I_t}{\dd t}}
= -\lim_{t \to \infty} \frac{\beta_t}{\alpha_t}
= 1.
\]
Thus, $\lim_{t\to \infty} \Var[\rho_t] = \lim_{t\to \infty} (q_t - m_t^2) = 1 - 1^2 = 0$ and it is readily seen that the convergence is exponentially fast since both $m_t$ and $q_t$ converge to $1$ exponentially fast. This immediately implies convergence in $W_2$ exponentially fast, because of the following estimates:
\[
W^2_2(\rho_t ,\delta_1)
\leq \EE |Z_t-1|^2
= q_t - 2\,m_t + 1
= \Var[\rho_t] + (1-m_t)^2.
\]
Finally, since $|Z_t| \leq 1$, we have $W_1( \rho_t , \delta_1) \leq \EE|1-Z_t| = 1 - m_t$, whence we deduce from \eqref{eq:mt} that
\[
W_1( \rho_t , \delta_1)
= \frac{2}{1+\frac{1+m_0}{1-m_0}\,\expo^{(\mu_+ - \mu_{-}) t}}
\leq 2\,\frac{1-m_0}{1+m_0}\, \expo^{-(\mu_+ - \mu_{-}) t}.
\]
\end{proof}

This concludes with the case $\mu_{-} < \mu_{+}$. From now on, we will assume that
\begin{equation}
	\label{Assumption:A1}
	\mu_- = \mu_+ \coloneqq \mu \in (0,1).
\end{equation}
The rest of this section is devoted to the large time analysis of the mean-field PDE under this assumption, which is trickier. First, notice that \eqref{Assumption:A1} and \eqref{eq:mt} imply that the mean opinion is conserved, i.e, $m_t \equiv m_0$ for all $t \geq 0$. Also, in this case the stationary distribution, denoted by $\rho_\infty$ will no longer be the Dirac mass at $1$; moreover, $\rho_\infty$  will actually depend on $\mu$. We leave the detailed analysis of $\rho_\infty$ for later, see Section \ref{sec:sec4}; for now, we only need to know that it exists.

We now establish a quantitative convergence guarantee for the solution of the mean-field PDE \eqref{eq:PDE_weak} under the assumption \eqref{Assumption:A1}. For this purpose, we first give a quick review of the so-called Fourier-based distance of order $s \geq 1$ (also known as Toscani distance) \cite{carrillo_contractive_2007}, defined by
\begin{equation}
\label{eq:Toscani_distance}
d_s(f,g) = \sup\limits_{\xi \in \mathbb{R}\setminus \{0\}} \frac{|\hat{f}(\xi)-\hat{g}(\xi)|}{|\xi|^s},
\end{equation}
where $f \in \mathcal{P}(\mathbb R)$ and $g \in \mathcal{P}(\mathbb R)$ are probability laws on $\mathbb R$, and
\[
\hat{f}(\xi)
\coloneqq \int_{\mathbb R} \expo^{-i\,x\,\xi}\,f(\dd x)
\] represents the Fourier transform of $f$. It is a well-known fact that (see \cite{carrillo_contractive_2007}) that $d_s(f,g) < \infty$ as long as $f$ and $g$ share the same moments up to order $\floor{s}$, where $\floor{s}$ denotes the integer part of $s$. We remark here that Fourier-based distances \eqref{eq:Toscani_distance} are introduced in a series of works \cite{carrillo_contractive_2007,gabetta_metrics_1995,goudon_fourier_2002} for the study of the problem of convergence to equilibrium for the spatially homogenous Boltzmann equation originated from statistical physics. These Fourier-based distances have also witnessed fruitful applications to novel sub-branches of traditional statistical physics, such as econophysics and sociophysics \cite{cao_equivalence_2024,during_boltzmann_2008,matthes_steady_2008,naldi_mathematical_2010}.

We now prove a convergence result in terms of these Fourier-based distances. To this end, note that taking $\varphi(x) =  \expo^{-i\,x\,\xi}$ in \eqref{eq:PDE_weak}, we obtain the following Fourier transformed version of the PDE:
\begin{align}
	\frac{\dd}{\dd t} \hat{\rho}_t(\xi)
	&= \int_{-1}^1 \expo^{- i \xi (1-\mu) x}
	\left[ \expo^{-i \xi \mu} \frac{1+m_0}{2} + e^{i \xi \mu} \frac{1-m_0}{2} \right] \rho_t(dx)
	- \hat{\rho}_t(\xi)
	\notag \\
	&= \left[\cos(\mu\,\xi)-i\,m_0\,\sin(\mu\,\xi)\right]\,\hat{\rho}_t((1-\mu)\,\xi) - \hat{\rho}_t(\xi).
	\label{eq:Fourier_PDE}
\end{align}

\begin{theorem}[Contraction in Fourier-based distances] \label{thm:2}
Assume that $\mu_- = \mu_+ \coloneqq \mu \in (0,1)$. Let $(\rho_t)_{t \geq 0}$ and $(\varrho_t)_{t \geq 0}$ be the solutions to \eqref{eq:PDE_weak} corresponding to initial datum $\rho_0$ and $\varrho_0$, respectively, with common first moment $m_0 \in (-1,1)$. Then, for each $s \geq 1$ we have for all $t\geq 0$:
\begin{equation}
	\label{eq:expo_decay_pairwise}
	d_s\left(\rho_t, \varrho_t \right)
	\leq d_s\left(\rho_0, \varrho_0\right)
	\,\expo^{-\left(1-(1-\mu)^s\right)\,t}.
\end{equation}
Consequently, for all $t\geq 0$,
\begin{equation}
\label{eq:expo_decay_Toscani}
d_s\left(\rho_t,\rho_\infty\right)
\leq d_s\left(\rho_0,\rho_\infty\right)\, \expo^{-\left(1-(1-\mu)^s\right)\,t}.
\end{equation}
\end{theorem}

\begin{proof} The proof is inspired from Theorem 2.2 in \cite{matthes_steady_2008}. For fixed $\xi \neq 0$, from \eqref{eq:Fourier_PDE} we have:
\begin{align*}
&\frac{\dd}{\dd t}\,\frac{\hat{\rho}_t(\xi) - \hat{\varrho}_t(\xi)}{|\xi|^s} + \frac{\hat{\rho}_t(\xi) - \hat{\varrho}_t(\xi)}{|\xi|^s} \\
&= \left[\cos(\mu\,\xi)-i\,m_0\,\sin(\mu\,\xi)\right]\,\frac{\hat{\rho}_t((1-\mu)\xi) - \hat{\varrho}_t((1-\mu)\xi)}{|\xi|^s} \\
&= \left[\cos(\mu\,\xi)-i\,m_0\,\sin(\mu\,\xi)\right]\,(1-\mu)^s\,\frac{\hat{\rho}_t((1-\mu)\xi) - \hat{\varrho}_t((1-\mu)\xi)}{|(1-\mu)\xi|^s}.
\end{align*}
Therefore, denoting $h_t(\xi) = \frac{\hat{\rho}_t(\xi) - \hat{\varrho}_t(\xi)}{|\xi|^s}$ and employing the elementary observation that $|\cos(\mu\,\xi)-i\,m_0\,\sin(\mu\,\xi)| \leq 1$, we deduce that
\[
\left\vert \frac{\dd}{\dd t}h_t(\xi) + h_t(\xi) \right\vert
\leq (1-\mu)^s\,\Vert h_t \Vert_\infty.
\]
Multiplying by $\expo^t$ and integrating, yields
\[
\left\vert \expo^t h_t(\xi) - h_0(\xi) \right\vert
\leq \int_0^t \left\vert \frac{\dd}{\dd r} \left(\expo^r h_r(\xi) \right) \right\vert \dd r
\leq (1-\mu)^s \int_0^t \expo^r \Vert h_r \Vert_\infty \dd r.
\]
Moving $\vert h_0(\xi)|$ to the right hand side and taking supremum over $\xi \neq 0$, we obtain
\[
\expo^t \Vert h_t \Vert_\infty
\leq \Vert h_0 \Vert_\infty + (1-\mu)^s \int_0^t \expo^r \Vert h_r \Vert_\infty \dd r.
\]
From Gronwall's lemma, we deduce that $\expo^t \Vert h_t \Vert_\infty
\leq \Vert h_0 \Vert_\infty \expo^{(1-\mu)^s t}$.
The advertised exponential decay result follows immediately since $\Vert h_t \Vert_\infty = d_s(\rho_t,\varrho_t)$.
\end{proof}

According to a classical result \cite{carrillo_contractive_2007} linking the Fourier-based distances $d_s$ and the Wasserstein distance $W_1$, from \eqref{eq:expo_decay_Toscani} we deduce the existence of a constant $C > 0$ depending only $s$ such that
\[
W_1\left(\rho_t,\rho_\infty\right)
\leq C\,[d_s\left(\rho_0,\rho_\infty\right) ]^{\frac{s-1}{s(2s-1)}} \,\expo^{-\frac{s-1}{s(2s-1)}\,\left(1-(1-\mu)^s\right)\,t}.
\]
Thus the solution of \eqref{eq:PDE_weak} converges exponentially fast to its equilibrium $\rho_\infty$ under the $W_1$ metric with rate $\frac{s-1}{s(2s-1)} (1-(1-\mu)^s)< \frac{\mu}{2}$. However, the following Theorem improves this rate to $\mu$, by working directly with the SDE description of the model using a coupling approach. To this end, notice that when $\mu_- = \mu_+ \coloneqq \mu \in (0,1)$, the SDE \eqref{eq:SDE_linear} becomes
\begin{equation}
	\label{eq:SDE_mu}
	\dd Z_t = \int_0^1 \mu ([ \mathbbm{1} \{ u < p\} - \mathbbm{1} \{ u \geq p\} ] - Z_{t^-})
	\mathcal{P}(\dd t, \dd u),
\end{equation}
where $p = (1+m_0)/2$. Since the intensity of $\mathcal{P}(\dd t, \dd u)$ is $\dd t \, \dd u$, we see that $\mathcal{B} \coloneqq \mathbbm{1} \{ u < p\} - \mathbbm{1} \{ u \geq p\}$ is a sample of a Rademacher random variable with parameter $p$, i.e., its distribution is $\beta = p\,\delta_1 + (1-p)\,\delta_{-1}$. Consequently, \eqref{eq:SDE_mu} can be written as
\begin{equation}
\label{eq:SDE_limit_rigorous_special_case}
\dd Z_t = \int_{\{-1,1\}} \mu\,\left(b - Z_{t^-}\right)\,\mathcal{Q}(\dd t,\dd b),
\end{equation}
where $\mathcal{Q}(\dd t,\dd b)$ denotes a Poisson point measure on $[0,\infty) \times \{-1,1\}$ with intensity $\dd t \, \beta(\dd b)$.

\begin{theorem}[Contraction in $W_1$]\label{thm:3}
Assume that $\mu_- = \mu_+ \coloneqq \mu \in (0,1)$. Let $(\rho_t)_{t \geq 0}$ and $(\varrho_t)_{t \geq 0}$ be the solutions to \eqref{eq:PDE_weak} corresponding to initial datum $\rho_0$ and $\varrho_0$, respectively, with common first moment $m_0 \in (-1,1)$. Then, we have for all $t\geq 0$:
\begin{equation}
	\label{eq:contraction_W1}
	W_1\left(\rho_t, \varrho_t \right)
	\leq W_1\left(\rho_0, \varrho_0\right)
	\,\expo^{-\mu t}.
\end{equation}
Consequently, for all $t\geq 0$,
\[
	W_1\left(\rho_t,\rho_\infty\right)
	\leq W_1\left(\rho_0,\rho_\infty\right)\, \expo^{-\mu \,t}.
\]
\end{theorem}

\begin{proof}
We use a coupling argument. Let $Z_t$ and $\tilde{Z}_t$ be the strong solutions to \eqref{eq:SDE_limit_rigorous_special_case} using exactly the same Poisson point measure $\mathcal{Q}$, and starting from initial conditions $Z_0 \sim \rho_0$ and $\tilde{Z}_0 \sim \varrho_0$ which are optimally coupled, that is, $\mathbb{E}\vert Z_0 - \tilde{Z}_0\vert = W_1(\rho_0 , \varrho_0)$. Let $h_t = \mathbb{E}\vert Z_t - \tilde{Z}_t \vert$. Then, for $0 \leq s \leq t$ we have
\begin{align*}
	&h_t - h_s \\
	&= \mathbb{E}\int_s^t  \int_{\{-1,1\}} \left[\left\vert\left\{Z_{r^-} + \mu (b-Z_{r^-}) \right\} -\{\tilde{Z}_{r^-} + \mu (b-\tilde{Z}_{r^-}) \}\right\vert - \left\vert Z_{r^-}-\tilde{Z}_{r^-}\right\vert\right] \mathcal{Q}(\dd r,\dd b) \\
	&= \mathbb{E}\int_s^t \left[\left\lvert (1-\mu)\,Z_{r}-(1-\mu)\,\tilde{Z}_{r}\right\lvert - \left\lvert Z_{r}-\tilde{Z}_{r}\right\lvert\right]\,\dd r \\
	&= -\mu\,\int_s^t h_r \,\dd r.
\end{align*}
Thus $h_t = h_0\,\expo^{-\mu t}$. Since $W_1$ is a coupling distance, we have $W_1(\rho_t, \varrho_t) \leq \mathbb{E}\vert Z_t - \tilde{Z}_t \vert = h_t$. The desired bound follows.
\end{proof}

\section{Stationary distribution of opinions}
\label{sec:sec4}
\setcounter{equation}{0}

We now turn our attention to the study of the stationary distribution $\rho_\infty$ of \eqref{eq:PDE_weak} in the case $\mu_{-} = \mu_{+} = \mu \in (0,1)$. We start by proving that it exists. Even though this can be easily achieved through a contraction argument, we can describe it more explicitly as follows. From \eqref{eq:SDE_limit_rigorous_special_case}, we see that a random variable $Z \sim \rho_\infty$ must solve the equation in distribution
\begin{equation}
\label{eq:characterization_of_equilibrium}
Z
\stackrel{\dd}{=} Z + \mu (\mathcal{B}-Z)
= (1-\mu)Z + \mu \mathcal{B},
\end{equation}
where $\mathcal{B}$ is a Rademacher distribution with parameter $p = (1+m_0)/2$, independent of $Z$, and $\stackrel{\dd}{=}$ stands for equality in the sense of distribution. To find a solution to \eqref{eq:characterization_of_equilibrium}, it is natural to consider the iteration $Z_{n+1} \coloneqq (1-\mu) Z_n + \mu \mathcal{B}_n$, starting from, say, $Z_0 \equiv 0$, where $(\mathcal{B}_n)_{n \in \NN}$ are i.i.d.\ copies of $\mathcal{B}$. This gives
\begin{equation}
\label{eq:Zn}
Z_n
= \mu \sum_{k=0}^{n-1} (1-\mu)^k \mathcal{B}_{n-1-k}.
\end{equation}
Now, one could study the limit of $\law(Z_n)$ as $n\to \infty$, which would certainly yield a distribution that solves \eqref{eq:characterization_of_equilibrium}. However, by reversing the indices of the $\mathcal{B}_k$'s in \eqref{eq:Zn}, we obtain a variable with the same law as $Z_n$, such that the sequence converges almost surely as $n \to \infty$. More specifically:

\begin{lemma}\label{lem:RV_characterization}
	With the previous notation, let
	\begin{equation}
		\label{eq:Zinfty}
		Z_\infty
		\coloneqq \mu \sum_{n=0}^\infty (1-\mu)^n \mathcal{B}_n.
	\end{equation}
	Then $Z_\infty$ is well defined, and it is a solution to \eqref{eq:characterization_of_equilibrium}. Consequently, $\rho_\infty \coloneqq \law(Z_\infty)$ is the unique stationary solution of \eqref{eq:PDE_weak} in the case $\mu_{-} = \mu_{+} = \mu \in (0,1)$. Moreover,
	\[
	\EE[Z_\infty] = m_0,
	\qquad
	\Var[Z_\infty] = \frac{\mu}{2-\mu} (1-m_0^2).
	\]
\end{lemma}

\begin{proof}
Since the $\mathcal{B}_n$'s are bounded and $1-\mu \in (0,1)$, the partial sums $\mu \sum_{k=0}^n (1-\mu)^k \mathcal{B}_k$ converge almost surely, thus $Z_\infty$ is well defined. Moreover, if $\mathcal{B}$ is an independent copy of the $\mathcal{B}_n$'s, then
\[
(1-\mu)Z_\infty + \mu \mathcal{B}
= \mu \sum_{n=0}^\infty (1-\mu)^{n+1} \mathcal{B}_n + \mu \mathcal{B} \\
= \mu \sum_{n=0}^\infty (1-\mu)^n \tilde{\mathcal{B}}_n
\stackrel{\dd}{=} Z_\infty,
\]
where $\tilde{\mathcal{B}}_0 \coloneqq \mathcal{B}$ and $\tilde{\mathcal{B}}_n \coloneqq \mathcal{B}_{n-1}$ for $n \geq 1$. Thus, $Z_\infty$ solves \eqref{eq:characterization_of_equilibrium}, which implies that $\rho_\infty \coloneqq \law(Z_\infty)$ is a stationary solution of \eqref{eq:PDE_weak}. Uniqueness of $\rho_\infty$ is a consequence of any of our contraction estimates \eqref{eq:expo_decay_pairwise} or \eqref{eq:contraction_W1}. The expressions for $\EE[Z_\infty]$ and $\Var[Z_\infty]$ follow from \eqref{eq:characterization_of_equilibrium} after a straightforward computation, which we omit.
\end{proof}

Thus, \eqref{eq:characterization_of_equilibrium} describes a two-parameter family of distributions: for each $\mu$ and $m_0$, there exists a unique solution $\rho_\infty \coloneqq \law(Z_\infty)$. From now on, for the ease of study and presentation, we restrict ourselves further (unless otherwise stated) to the special case where
\begin{equation}
	\label{Assumption:A2}
	m_0 = \int_{-1}^1 x \rho_\infty(dx) = 0.
\end{equation}
We remark that in this case, the infinite (random) series $\sum_n (1-\mu)^n \mathcal{B}_n = Z_\infty / \mu$ is termed the \emph{Bernoulli convolution}, which has been studied extensively \cite{erdos_family_1939,solomyak_random_1995,varju_recent_2016,varju_absolute_2019,wintner_convergent_1935}. The following characterization of $\rho_\infty$ is well known in this literature; here we provide a different proof using the PDE \eqref{eq:Fourier_PDE}:

\begin{corollary}\label{cor:integral_representation}
	Assume that $\mu_- = \mu_+ \coloneqq \mu \in (0,1)$ and that $m_0 = 0$. Then
	\begin{equation}
		\label{eq:Fourier_representation}
		\hat{\rho}_\infty(\xi) = \prod_{n=0}^\infty \cos\left(\mu\,(1-\mu)^n\,\xi\right).
	\end{equation}
	Consequently, if $\rho_\infty$ admits a continuous density, then \begin{equation}
		\label{eq:inverse_Fourier_representation}
		\rho_\infty(x) = \frac{1}{2\pi}\,\int_\mathbb{R} \prod_{n=0}^\infty \cos\left(\mu\,(1-\mu)^{n}\,\xi\right)\,\expo^{i\,x\,\xi}\,\dd \xi.
	\end{equation}
\end{corollary}

\begin{proof}
Inserting $m_0 = 0$ into the Fourier transformed PDE \eqref{eq:Fourier_PDE} and sending $t \to \infty$, we obtain \[\hat{\rho}_\infty(\xi) = \cos(\mu\,\xi)\,\hat{\rho}_\infty((1-\mu)\,\xi) = \prod_{n=0}^\infty \cos\left(\mu\,(1-\mu)^{n}\,\xi\right)\] via iteration, using the fact that $\hat{\rho}_\infty(0)= 1$ (alternatively, just compute $\EE[\expo^{-i \xi Z_\infty}]$ in \eqref{eq:Zinfty}). As a result, the representation \eqref{eq:inverse_Fourier_representation} for the equilibrium distribution $\rho_\infty$ follows from the celebrated Fourier inversion formula.
\end{proof}

So far we have proven that the stationary distribution $\rho_\infty$ exists, and we obtained (somewhat explicit) general ways of describing it, either via a ``random variable characterization'' as in Lemma \ref{lem:RV_characterization} or through an integral representation as demonstrated in Corollary \ref{cor:integral_representation}. However, as we shall see, the nature of $\rho_\infty$ changes wildly (depending on the specific value of $\mu$), which indicates the existence of the so-called {\it phase transition} phenomenon. In order to investigate $\rho_\infty$ more explicitly, we split our analysis in three cases: $\mu = 1/2$, $\mu>1/2$, and $\mu<1/2$.

\subsection{Case $\mu=1/2$: uniform distribution}

When $\mu = 1/2$ and $m_0 = 0$, the infinite product appearing in \eqref{eq:Fourier_representation} is amenable to explicit evaluation and yields that
\[\hat{\rho}_\infty(\xi) = \prod_{n=1}^\infty \cos\left(\frac{\xi}{2^n}\right) = \frac{\sin \xi}{\xi},\] which coincides with the Fourier transform of the uniform distribution on $[-1,1]$. (On the other hand, as long as $\mu \neq 1/2$, it is prohibitively hard (if possible at
all) to calculate the integral \eqref{eq:inverse_Fourier_representation} in a closed form). We obtain:

\begin{proposition}\label{prop:uniform_equilibrium}
	Assume that $\mu_- = \mu_+ = 1/2$ and that $m_0 = 0$. Then $\mathrm{Uniform}([-1,1])$ is the unique equilibrium solution of the mean-field PDE \eqref{eq:PDE_weak}.
\end{proposition}

In the setting of the Bernoulli convolution, this Proposition corresponds to the fact that $Z_\infty / \mu = \sum_n (1-\mu)^n \mathcal{B}_n$ is uniform on $[-2,2]$ for $\mu=1/2$. Alternatively, the following is a heuristic (but more illustrative) argument as to why $\rho_\infty = \mathrm{Uniform}([-1,1])$. Indeed, when $\mu = 1/2$, the equation in distribution \eqref{eq:characterization_of_equilibrium} simplifies to
\begin{equation}
\label{eq:Zmu05}
Z
\stackrel{\dd}{=}
\tfrac{1}{2}Z + \tfrac{1}{2} \mathcal{B}.
\end{equation}
When $Z \sim \mathrm{Uniform}([-1,1])$, we have $\tfrac{1}{2}Z \sim \mathrm{Uniform}([-\frac{1}{2},\frac{1}{2}])$. The additive term $+\frac{1}{2}\mathcal{B}$ transforms the interval $[-\frac{1}{2},\frac{1}{2}]$ into either $[-1,0]$ or $[0,1]$ with equal probabilities. Thus, $\tfrac{1}{2}Z + \tfrac{1}{2} \mathcal{B}$ again has the distribution $\mathrm{Uniform}([-1,1])$. This compelling intuition is illustrated in Figure \ref{fig:mu=1/2} below.

\begin{figure}[!htb]
  \centering
  \includegraphics[scale = 0.75]{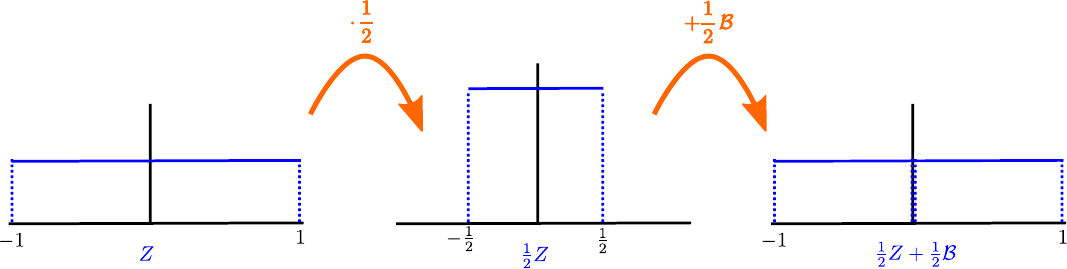}
  \caption{Geometric illustration of the fact that $Z \sim \mathrm{Uniform}([-1,1])$ satisfies the relation \eqref{eq:Zmu05}.}
  \label{fig:mu=1/2}
\end{figure}

\subsection{Case $\mu > 1/2$: emergence of fractal structures}

Now we turn to the case when $\mu \in (1/2, 1)$. In order to determine the random variable $Z_\infty$ from the relation \eqref{eq:characterization_of_equilibrium}, we call the map
\begin{equation}
	\label{eq:procedure}
	Z \mapsto (1-\mu)\,Z + \mu\,\mathcal{B}
\end{equation}
as $\varphi_\mu(Z)$. As we are working with the scenario where $\mu > \frac 12$, it is easy to see that if $\textrm{supp}(Z) \subset [-1,1]$, then the relation \eqref{eq:characterization_of_equilibrium} yields
\[
\textrm{supp}\left((1-\mu)\,Z + \mu\,\mathcal{B}\right) = \textrm{supp}(Z) \subset [-1,1-2\mu] \cup [2\mu-1,1].
\]
Therefore, a recursive argument should lead to the conclusion that $\textrm{supp}(Z_\infty)$ is a Cantor-like set. For instance, in the special case when $\mu = \frac 23$ so that $1-(2\mu-1) = (2\mu-1)-(1-2\mu)$, we expect that $Z_\infty \sim \mathrm{Uniform}(2\cdot \mathcal{C} - 1)$, in which the set $\mathcal{C}$ denotes the classical Cantor ternary set on $[0,1]$. In the general case where $\mu \in (1/2, 1)$ but $\mu \neq 2/3$, the set $2\cdot \mathcal{C} - 1$ will be replaced by the ``limit set'' (denoted by $\mathcal{C}_\infty$) one gets after applying the procedure \eqref{eq:procedure} infinitely many times (starting from a zero mean random variable whose support is contained in $[-1,1]$) and extracting the support of the resulting random variable. See Figure \ref{fig:illustration_map} below for an illustration of this procedure.

\begin{figure}[!htb]
	\centering
	\includegraphics[scale = 0.8]{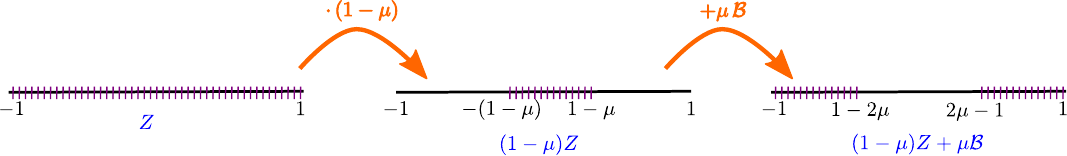}
	\caption{Illustration of the map $\varphi_\mu$ for $\mu > \frac 12$. It is clear that the support of $(1-\mu)\,Z + \mu\,\mathcal{B}$ is a subset of $[-1,1-2\mu] \cup [2\mu-1,1]$ if $\textrm{supp}(Z) \subset [-1,1]$.}
	\label{fig:illustration_map}
\end{figure}

We now turn these intuitive arguments into a rigorous statement, summarized as follows:

\begin{proposition}\label{prop:Cantor-like_equili}
Assume that $\mu \in (1/2,1)$ and $m_0 = 0$. Then, $\rho_\infty = \mathrm{Uniform}(\mathcal{C}_\infty)$, where the (Cantor-like) set $\mathcal{C}_\infty$ is constructed via the following recursive procedure: set $\mathcal{C}_0 = [-1,1]$ and for $n \in \mathbb{N}$ define
\[
\mathcal{C}_{n+1} = \{(1-\mu)\,\mathcal{C}_n - \mu\} \cup \{(1-\mu)\,\mathcal{C}_n + \mu\},
\]
which is a proper closed subset of $\mathcal{C}_n$, and then set $\mathcal{C}_\infty \coloneqq \bigcap_{n\geq 0} \mathcal{C}_n$. The distribution $\mathrm{Uniform}(\mathcal{C}_\infty)$ is understood as the limit of $\mathrm{Uniform}(\mathcal{C}_n)$ when $n\to \infty$. In particular, this means that the equilibrium distribution of the mean-field PDE \eqref{eq:PDE_weak} under the assumptions \eqref{Assumption:A2} and $\mu \in (1/2,1)$ is a singular measure whose support is on a set of Lebesgure measure zero.
\end{proposition}

A technical proof of Proposition \eqref{prop:Cantor-like_equili} in the setting of the Bernoulli convolutions can be found in \cite{kershner_symmetric_1935}, and for the sake of completeness and for the reader's convenience, we present a different and elementary proof in the Appendix.

\begin{remark}
A straightforward computation using self-similarity allows us to compute the Hausdorff dimension of $\mathcal{C}_\infty$ as well, which equals $\ln(2) / \ln\left(1/(1-\mu)\right)$. In particular, when $\mu = 2/3$, one recovers (as expected) the Hausdorff dimension of the standard Cantor set $\mathcal{C}$.
\end{remark}

\subsection{Case $\mu < 1/2$: partial results and asymptotic normality}

Now we turn our attention to the case when $\mu \in (0,1/2)$. Unfortunately, in this scenario we fail to figure out the distribution $\rho_\infty \coloneqq \law(Z_\infty)$ solving \eqref{eq:characterization_of_equilibrium} in general. It turns out for $\mu \in (0,1/2)$ the story is much more complicated and non-trivial. One central question of interest lies in the possibility of proving the absolute continuity of the equilibrium distribution $\rho_\infty$ (with respect to the usual Lebesgue measure on $[-1,1]$). It has been proved among the literature on Bernoulli convolution \cite{solomyak_random_1995} that $\rho_\infty$ is absolutely continuous for a.e. $\mu \in (0,1/2)$. Erd{\"o}s \cite{erdos_family_1939} constructed a countable set of (inverse Pisot) numbers $\mu \in (0,1/2)$ such that $\rho_\infty$ is singular, which remain to be the only known exceptions. On the other hand, a complete characterization of explicit examples of $\mu \in (0,1/2)$ for which $\rho_\infty$ is absolutely continuous still remains open, and so far only very few such explicit examples are known \cite{varju_absolute_2019,wintner_convergent_1935}. In particular, Wintner \cite{wintner_convergent_1935} demonstrated the absolute continuity of $\rho_\infty$ when $\mu$ is of the form $\mu = 1-2^{-1/k}$ for $k\in \mathbb{N}_+$ and Varj{\'u} \cite{varju_absolute_2019} showed the absolute continuity of $\rho_\infty$ when $\mu$ belongs to a class of algebraic numbers satisfying a list of technical conditions.

We aim to report some partial results along this direction. First, we demonstrate the explicit distribution of $Z_\infty$ (under the usual centralized first moment assumption \eqref{Assumption:A2}) for one particular choice of $\mu \in (0,1/2)$. Although this choice of $\mu$ can be handled using the aforementioned general result proved by Wintner \cite{wintner_convergent_1935}, our proof is rather elementary and the underlying geometric intuition is highlighted as well.

\begin{corollary}\label{cor:example_mu}
Assume that $\mu = 1 - 1/\sqrt{2} \approx 0.29289$ and $m_0 = 0$. Then the distribution of $Z_\infty$ satisfying \eqref{eq:characterization_of_equilibrium} is given by the following volcano-shaped density
\begin{equation}\label{eq:rho_equil_example}
\rho_\infty(x) \coloneqq \begin{cases}
\frac{1}{1+r}\,\frac{1+x}{1-r}, ~& -1\leq x\leq -r,\\
\frac{1}{1+r}, ~& -r\leq x \leq r,\\
\frac{1}{1+r}\,\frac{1-x}{1-r}, ~& r\leq x\leq 1,
\end{cases}
\end{equation}
in which $r = 4\,\mu - 1 \approx 0.17157$.
\end{corollary}

\begin{proof} The proof follows merely from a straightforward computation (albeit lengthy and tedious) that we will omit, all we need to check is that the density function \eqref{eq:rho_equil_example} indeed satisfies $Q^*[\rho_\infty] = 0$ \eqref{eq:operator_Q*} when $\mu = 1 - 1/\sqrt{2}$ and $m_t \equiv 0$. The essential geometric intuition is illustrated in Figure \ref{fig:example_volcano}.

\begin{figure}[!htb]
  \centering
  \includegraphics[scale = 0.75]{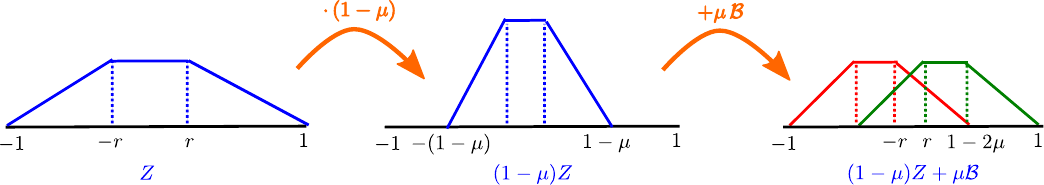}
  \caption{Geometric sketch of the observation that $Z \sim \rho_\infty$ \eqref{eq:rho_equil_example} fulfills the relation \eqref{eq:characterization_of_equilibrium} when $\mathbb{E}[Z] = 0$, $\mu = 1 - 1/\sqrt{2}$, and $r = 4\,\mu - 1$. The key idea lies in the proper choice of $\mu$ and $r$ so that the ``concatenation'' of two Volcano-shaped distributions yields the target Volcano-shaped density $\rho_\infty$ defined by \eqref{eq:rho_equil_example}.}
  \label{fig:example_volcano}
\end{figure}

In order for the designed density \eqref{eq:rho_equil_example} to be the distribution of $Z_\infty$ satisfying \eqref{eq:characterization_of_equilibrium}, we deduce the following conditions on $r$ and $\mu$:
\[
r\,(1-\mu) + \mu = 1 - 2\,\mu \quad \textrm{and} \quad -r\,(1-\mu) + \mu = r.
\]
Solving this system leads us to $\mu = 1 - 1/\sqrt{2}$ and $r = 4\,\mu - 1$.
\end{proof}

Next, we present several numerical experiments of the opinion dynamics when the parameters are chosen such that $\mu \in (0,1/2)$ and $m_0 = 0$, from the (stochastic) agent-based point of view along with the (mean-field) PDE perspective (when the number of agents tends to infinity), see Figure \ref{fig:numerics}. For the simulation results reported below, we always use the uniform distribution on $[-1,1]$ as the initial datum. We demonstrate the agent-based simulation results with $N = 5\cdot 10^6$ agents for three cases: $\mu = 0.25$, $\mu = 1 - 1/\sqrt{2}$, and $\mu = 0.4$. In each case, we display the histogram of the $N$ agents, scaled vertically to approximate the density. On the other hand, we display the evolution of the numerical solution $\rho_t$ of the Boltzmann-type equation \eqref{eq:PDE} at various time instants, using the standard fourth-order Runge-Kutta scheme with time step $\Delta t = 0.01$ and spatial discretization $\Delta x = 0.0001$, for the same three values of $\mu$, side-by-side to the agent-based simulation results. In both cases, we ran the simulation up to time $t=20$, which seems enough to have reached equilibrium. It can be observed that the outcomes of the agent-based simulations (at equilibrium) agree well with their mean-field counterpart (predicted by the equilibrium solution of the Boltzmann-type PDE \eqref{eq:PDE}).

\def\figscale{0.58}
\def\figwidth{0.47\textwidth}
\begin{figure}[!htb]
	\begin{subfigure}{\figwidth}
		\centering
		\includegraphics[scale=\figscale]{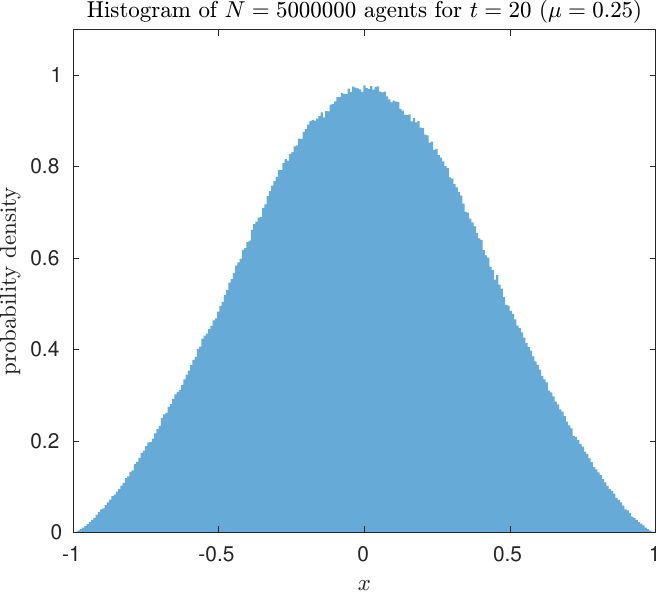}
	\end{subfigure}
	\hspace{0.1in}
	\begin{subfigure}{\figwidth}
		\centering
		\includegraphics[scale=\figscale]{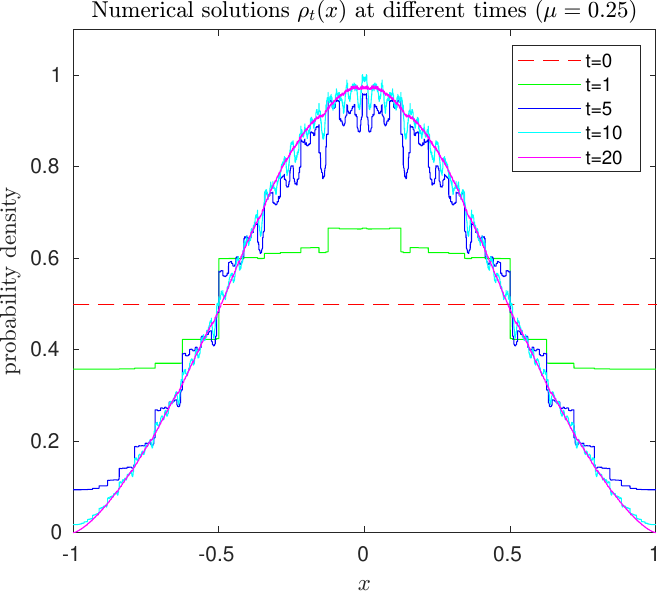}
	\end{subfigure}
	\\
	\begin{subfigure}{\figwidth}
		\centering
		\includegraphics[scale=\figscale]{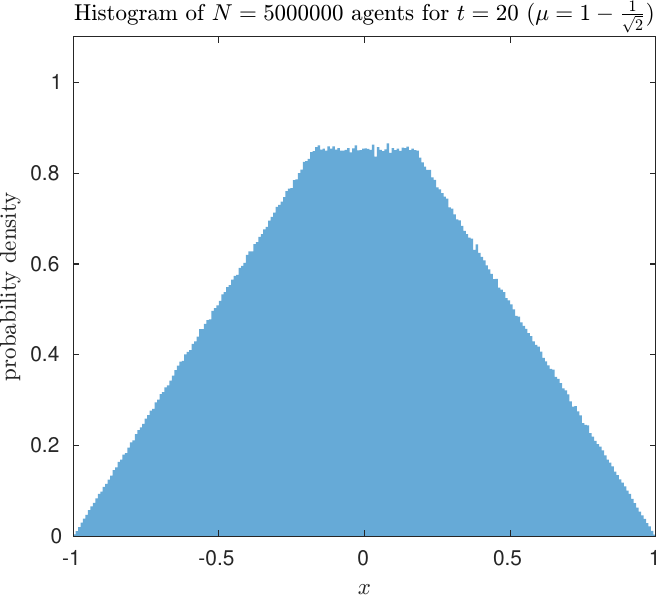}
	\end{subfigure}
	\hspace{0.1in}
	\begin{subfigure}{0.45\textwidth}
		\centering
		\includegraphics[scale=\figscale]{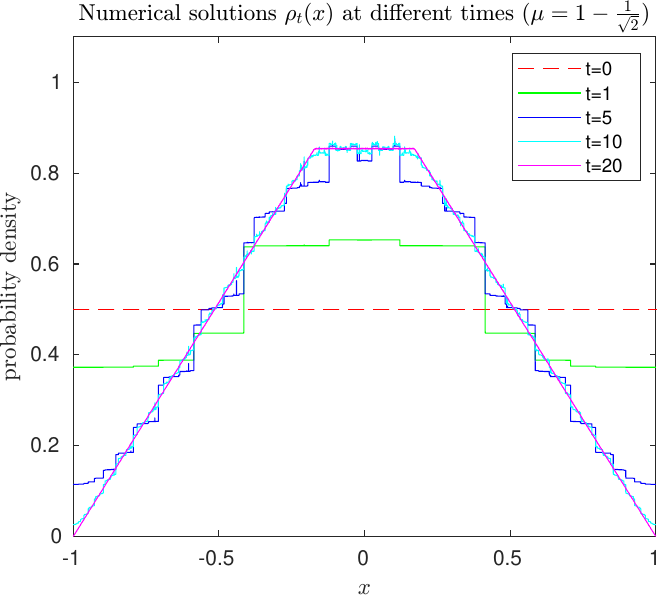}
	\end{subfigure}
    \\
    \begin{subfigure}{\figwidth}
    	\centering
    	\includegraphics[scale=\figscale]{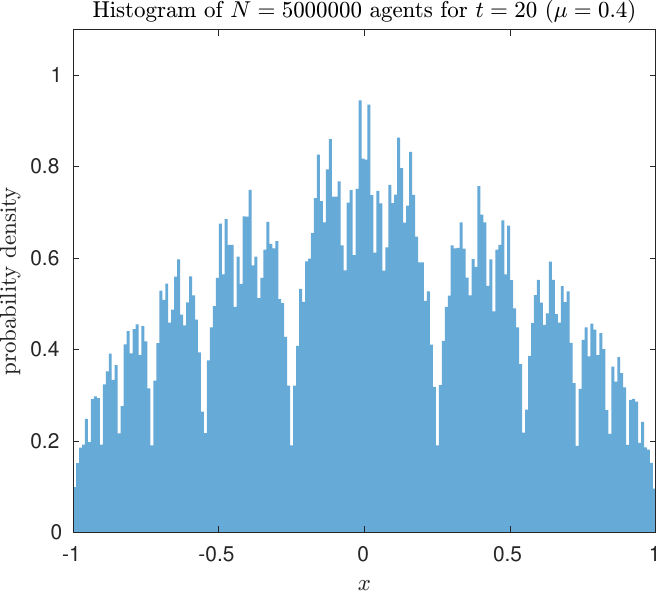}
    \end{subfigure}
    \hspace{0.1in}
    \begin{subfigure}{\figwidth}
    	\centering
    	\includegraphics[scale=\figscale]{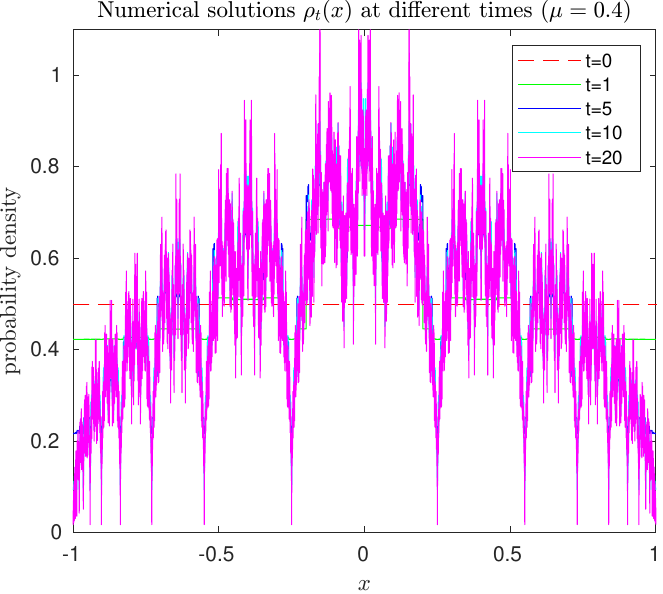}
    \end{subfigure}
	
	\caption{{\bf Left}: Simulations of agent-based model with $N = 5\cdot 10^6$ agents, and $\mu = 0.25$ (top), $\mu=1-1/\sqrt{2}$ (middle), $\mu = 0.4$ (bottom). {\bf Right}: Evolution of the solution $\rho_t$ to the Boltzmann-type equation \eqref{eq:PDE} with respect to time, again with and $\mu = 0.25$ (top), $\mu=1-1/\sqrt{2}$ (middle), $\mu = 0.4$ (bottom).}
	\label{fig:numerics}
\end{figure}

We now claim that for $\mu \ll 1$ small enough, the equilibrium distribution $\rho_\infty$ resembles a Gaussian density. More precisely, after normalizing $Z_\infty$ by its standard deviation $\sigma = \sqrt{\mu/(2-\mu)}$, we provide a quantitative convergence result showing that the law of $Z_\infty/\sigma$ converges to the standard Gaussian density as $\mu \to 0$.

\begin{theorem}\label{prop:Gaussian_approximation}
Assume that $m_0 = 0$. Then, $Z_\infty / \sigma$ converges in distribution to the standard Gaussian $\mathcal{N}(0,1)$ as $\mu \to 0$. Moreover, there exists some constant $C > 0$ (independent of $\mu$) such that
\begin{equation}\label{eq:d4_bound}
d_4\left(\textrm{Law}(Z_\infty/\sigma),\mathcal{N}(0,1)\right) \leq C\,\mu
\end{equation}
for all small enough $\mu$.
\end{theorem}

\begin{proof}
Call $f = f_\mu = \law(Z_\infty/\sigma)$. Since $Z_\infty$ is given by \eqref{eq:Zinfty}, a similar computation as in the derivation of
\eqref{eq:Fourier_representation} gives rise to
\[\hat{f}(\xi)
= \prod_{n=0}^\infty \cos\left(\sqrt{\mu\,(2-\mu)}\,(1-\mu)^n\,\xi\right).
\]
Therefore, in order to show that $f$ converges to the standard Gaussian density as $\mu \to 0$, it suffices to prove that
\begin{equation}\label{eq:convergence_ln_Fourier}
\ln \hat{f}(\xi) = \sum\limits_{n = 0}^\infty \ln \cos\left(\sqrt{\mu\,(2-\mu)}\,(1-\mu)^n \,\xi\right) \xrightarrow{\mu \to 0} -\frac{\xi^2}{2}.
\end{equation}
Thanks to the Taylor expansion of the function $x \mapsto \ln\cos(x)$ around $x = 0$, i.e., $\ln\cos(x) = -\frac{x^2}{2} + \mathcal{O}(x^4)$, we deduce for $\mu \ll 1$ that
\begin{align*}
\sum_{n = 0}^\infty \ln \cos\left(\sqrt{\mu\,(2-\mu)}\,(1-\mu)^n\,\xi\right)
&= -\frac{\xi^2}{2}\,\mu\,(2-\mu)\,\sum_{n = 0}^\infty (1-\mu)^{2n} \\
&\qquad + \xi^4 \,\mathcal{O}\left(\mu^2\,(2-\mu)^2 \sum_{n = 0}^\infty (1-\mu)^{4n} \right) \\
&= -\frac{\xi^2}{2} + \xi^4\,\mathcal{O}\left(\frac{\mu^2\,(2-\mu)^2}{1-(1-\mu)^4}\right).
\end{align*}
Consequently, the advertised asymptotic behavior \eqref{eq:convergence_ln_Fourier} is established by noting that $\frac{\mu^2\,(2-\mu)^2}{1-(1-\mu)^4} = \mathcal{O}(\mu)$ for small enough $\mu$ (using L'H\^ospital's rule).

We now proceed to the proof of the quantitative estimate \eqref{eq:d4_bound} using the Toscani distance $d_4$. We denote by $\mathcal{N}(x)$ the density function of the standard Gaussian random variable and note that its Fourier transform is given by $\hat{\mathcal{N}}(\xi) = \expo^{-\xi^2/2}$ for all $\xi \in \mathbb R$. We recall that
\[
d_4(f,\mathcal{N}) = \sup\limits_{\xi \in \mathbb{R}\setminus \{0\}} \frac{|\hat{f}(\xi)-\hat{\mathcal{N}}(\xi)|}{|\xi|^4}.
\]
Since $|\hat{f}(\xi)| \leq 1$ and $|\hat{\mathcal{N}}(\xi)|\leq 1$ for all $\xi$, for an arbitrary but fixed $R > 0$ (whose value remains to be determined) we have
\begin{equation}\label{eq:d4_bound2}
d_4(f,\mathcal{N}) \leq \sup\limits_{\xi \in [-R,R] \setminus \{0\}} \frac{|\hat{f}(\xi)-\hat{\mathcal{N}}(\xi)|}{|\xi|^4} + \frac{1}{R^4}.
\end{equation}
Now we bound $|\hat{f}(\xi)-\hat{\mathcal{N}}(\xi)|$ as
\begin{equation*}
|\hat{f}(\xi)-\hat{\mathcal{N}}(\xi)| = \left|\expo^{\ln \hat{f}(\xi) + \frac{\xi^2}{2}} - 1\right|\,\expo^{-\frac{\xi^2}{2}} \leq \left|\expo^{C\,\xi^4} - 1\right| \leq C\,\expo^{C\,R^4}\,\xi^4
\end{equation*}
in which $\xi \in [-R,R]$ and $C = C_\mu = \mathcal{O}\left(\frac{\mu^2\,(2-\mu)^2}{1-(1-\mu)^4}\right) = \mathcal{O}(\mu)$ for all small enough $\mu$. Therefore, we deduce from \eqref{eq:d4_bound2} that
\[
d_4(f,\mathcal{N}) \leq C\,\expo^{C\,R^4} + \frac{1}{R^4} \leq 4\,C
\]
by choosing $R = C^{-\frac 14}$. This completes the proof.
\end{proof}

\section{Conclusion and future work}\label{sec:sec5}
\setcounter{equation}{0}

In this manuscript, we proposed and analyzed a novel opinion-dynamics model in the mean-field regime, i.e., as the number of agents $N$ tends to infinity. We proved several quantitative convergence guarantees for the solution of the mean-field Boltzmann-type PDE to its unique equilibrium distribution, and we demonstrated that this distribution of opinions depends heavily on the model parameters. Surprisingly, for certain regions of parameter choice we managed to prove the emergence of Cantor-like fractal structures, which provide a mathematically rigorous explanation for the so-called opinion fragmentation phenomenon. Our model also bridges interacting multi-agents systems and kinetic-type PDEs with the Bernoulli convolution, which indicates its inter-disciplinary nature and suitability for extensive subsequent research efforts. We provided numerical simulation, both for the stochastic agent-based system and the mean-field Boltzmann-type PDE.

The present article also leaves many important problems to be addressed in future works. First, can we employ tools from evolution PDEs to offer a potential alternative proof to a number of technical results on Bernoulli convolutions? For instance, we believe that it would be very interesting (perhaps also rather challenging) if one can prove the absolute continuity of $\rho_\infty$ for a.e. $\mu_+ = \mu_{-} \in (0,1/2)$ purely from a PDE viewpoint, without resorting to advanced probabilistic and information-theoretic techniques. Second, we aim to analyze our opinion model in the case where $\mu_+ = \mu_{-}$ and $m_0 \neq 0$ in a future paper, and explore the influence of the non-zero initial average opinion (where no symmetry around the space of admissible opinions $[-1,1]$ is maintained anymore) on the shape of the equilibrium distribution of opinions. Lastly, it would be very desirable to prove \emph{uniform} propagation of chaos for this model, i.e., with explicit convergence rates in $N$ that do not depend on time. Simulations of the multi-agent system presented in this work seem to support this conjecture.

\section*{Appendix}
\setcounter{equation}{0}

\subsection*{Proof of Proposition \ref{prop:Cantor-like_equili}}

\begin{proof}
The proof follows readily from the construction process of the set $\mathcal{C}_\infty$. Indeed, let $\mathcal{C}_n$ denote the set obtained at the $n$-stage of the construction of $\mathcal{C}_\infty$ and call $\lambda_n = \mathrm{Uniform}(\mathcal{C}_n)$ be the uniform distribution on $\mathcal{C}_n$, then $\lambda_\infty \coloneqq \mathrm{Uniform}(\mathcal{C}_\infty) = \lim_{n \to \infty} \lambda_n$ where the convergence is understood in the sense of distribution. Now, taking $Z_n \sim \lambda_n$ for each $n \in \mathbb N$ so that $Z_\infty \sim \lambda_\infty$ is the limit of $Z_n$ as $n\to \infty$, we can easily see that the law of $(1-\mu)\,Z_n + \mu\,\mathcal{B}$ coincides with $\lambda_{n+1}$, hence in a distribution sense, \[\lim_{n \to \infty} \textrm{Law}(Z_n) = \lim_{n \to \infty} \textrm{Law}\left((1-\mu)\,Z_n + \mu\,\mathcal{B}\right). \] This implies that $Z_\infty$ satisfies the relation \eqref{eq:characterization_of_equilibrium}. In order to show that the set $\mathcal{C}_\infty$ has zero Lebesgue measure (denoted by $\lambda$), it suffices to notice that
\begin{equation*}
\begin{aligned}
\lambda(\mathcal{C}_\infty) &= \lambda([-1,1]) - \left[2\,(2\,\mu-1) + 4\,(1-\mu)\,(2\,\mu-1) + 8\,(1-\mu)^2\,(2\,\mu-1) + \cdots\right] \\
&= 2 - (2\,\mu-1)\,\sum\limits_{n=1}^\infty 2^n\,(1-\mu)^{n-1} = 0.
\end{aligned}
\end{equation*}
Thus the proof of Proposition \ref{prop:Cantor-like_equili} is completed.
\end{proof}

\end{document}